\documentclass[11pt,reqno]{amsart}
\usepackage{hyperref}
\usepackage{url}
\usepackage{amssymb}
\usepackage{diagbox}
\usepackage[applemac]{inputenc} 
\usepackage{array}
\newtheorem{theorem}{Theorem}

\newtheorem{lemma}[theorem]{Lemma}

\theoremstyle{definition}

\theoremstyle{remark}
\newtheorem{rem}{Remark}
\numberwithin{equation}{section}
\numberwithin{theorem}{section}
\numberwithin{conj}{section}
\numberwithin{defn}{section}

\allowdisplaybreaks[4]
\begin{document}

\title[Congruences for 3- and 9-Colored Generalized Frobenius Partitions]
 {Congruences Modulo Powers of 3 for 3- and 9-Colored Generalized Frobenius Partitions}

\author{Liuquan Wang}
\address{School of Mathematics and Statistics, Wuhan University, Wuhan 430072, Hubei, People's Republic of China}
\email{wanglq@whu.edu.cn; mathlqwang@163.com}



\dedicatory{}

\maketitle

\begin{abstract}
Let $c\phi_{k}(n)$ be the number of $k$-colored generalized Frobenius partitions of $n$. We establish some infinite families of congruences for $c\phi_{3}(n)$ and $c\phi_{9}(n)$ modulo arbitrary powers of 3, which refine the results of Kolitsch. For example, for $k\ge 3$ and $n\ge 0$, we prove that
\[c\phi_{3}\Big(3^{2k}n+\frac{7\cdot 3^{2k}+1}{8}\Big) \equiv 0 \pmod{3^{4k+5}}.\]
We give two different proofs to the congruences satisfied by $c\phi_{9}(n)$.  One of the proofs uses an relation between $c\phi_{9}(n)$ and $c\phi_{3}(n)$ due to Kolitsch, for which we provide a new proof in this paper.
\end{abstract}

\section{Introduction}
In 1984, Andrews \cite{Andrews} introduced the concept of $k$-colored generalized Frobenius partitions.
 We first color all the nonnegative integers using ``colors'' denoted by $1,2,\cdots, k$. Then we impose an ordering on these colored integers as follows:
\begin{equation}\label{new-ordering}0_1\prec 0_2\prec \cdots \prec 0_k\prec 1_1\prec 1_2\prec \cdots \prec 1_k \prec 2_1\prec 2_2\prec \cdots \prec 2_k\prec \cdots.\end{equation}
Here ``$\prec$'' is used to differentiate the inequality from the usual inequality ``$<$''. A $k$-colored generalized Frobenius partition
is a two-row array of colored integers of the form
\begin{align*}
\begin{pmatrix}
a_{1},a_{2},\cdots, a_{r} \\
b_{1},b_{2},\cdots, b_{r}
\end{pmatrix}
\end{align*}
such that
\begin{align}
a_r \prec a_{r-1}\prec \cdots \prec a_1, \quad b_{r}\prec b_{r-1}\prec \cdots \prec b_1.
\end{align}
The number being partitioned by this partition is
\[n=r+\sum_{i=1}^{r}(a_i+b_i).\]

For any positive integer $k$, Andrews \cite{Andrews} used the symbol $c\phi_{k}(n)$ to denote the number of $k$-colored generalized Frobenius partitions of $n$. Let
\begin{align*}
\mathrm{C}\Phi_{k}(q):=\sum_{n=0}^{\infty}c\phi_{k}(n)q^n.
\end{align*}
Andrews \cite{Andrews} proved that
\begin{align}
\mathrm{C}\Phi_{k}(q)=\frac{1}{(q;q)_{\infty}^k}\sum\limits_{m_1, \cdots, m_{k-1}=-\infty}^{\infty}{q^{Q(m_1,\cdots, m_{k-1})}},
\end{align}
where
\[Q(m_1,\cdots, m_{k-1})=\sum\limits_{i=1}^{k-1}{m_{i}^2}+\sum\limits_{1\le i <j \le k-1}{m_im_j}\]
and
\[(a;q)_{\infty}=\prod\limits_{n=0}^{\infty}(1-aq^n), \quad |q|<1.\]

To investigate arithmetic properties of $c\phi_{k}(n)$, Andrews \cite{Andrews} obtained alternative representations for $\mathrm{C}\Phi_{k}(q)$ with $k\in \{2, 3, 5\}$. In particular, for $k=3$, he \cite[Eq.\ (9.4)]{Andrews} proved that
\begin{align}\label{gen}
\sum\limits_{n=0}^{\infty}{c\phi_{3}(n)q^n}&=\frac{1}{(q;q)_{\infty}^{3}}\left(1+6\sum_{n=0}^{\infty}\left(\frac{q^{3n+1}}{1-q^{3n+1}}-\frac{q^{3n+2}}{1-q^{3n+2}} \right)\right).
\end{align}
From the formulas of $\mathrm{C}\Phi_{k}(q)$, Andrews \cite{Andrews} found some nice properties of $c\phi_{k}(n)$. For instance, he proved that for $n\ge 0$,
\begin{align}\label{cphi-2-mod5}
c\phi_{2}(5n+3) \equiv 0 \pmod{5}.
\end{align}
Since then, many congruences satisfied by $c\phi_{k}(n)$ have been discovered. Sellers \cite{Sellers} conjectured that \eqref{cphi-2-mod5} can be extended to a congruence modulo arbitrary powers of 5. Namely, for any integers $k \ge 1$ and $n\ge 0$, he conjectured that
\begin{align}
c\phi_{2}\left(5^{k}n+\lambda_{k} \right)\equiv 0 \pmod{5^k},
\end{align}
where $\lambda_{k}$ is the reciprocal of 12 modulo $5^{k}$. This conjecture was later proved by Paule and Radu \cite{Paule} using the theory of modular forms.

After the work of Andrews, Kolitsch \cite{Kolitsch,Kolitsch1989} introduced the function $\overline{c\phi}_{k}(n)$, which denotes the number of $k$-colored generalized Frobenius partitions of $n$ whose order is $k$ under cyclic permutation of the colors. He \cite{Kolitsch1989} proved that for any positive integer $m$,
\begin{align}\label{bar-general}
\overline{c\phi}_{m}(n)=\sum_{d|(m,n)}\mu(d)c\phi_{\frac{m}{d}}\left(\frac{n}{d} \right),
\end{align}
where $\mu(x)$ is the M\"obius function. In particular, when $m$ is a prime, we have
\begin{align}\label{bar-relate}
\overline{c\phi}_{m}(n)=c\phi_{m}(n)-p(\frac{n}{m}),
\end{align}
where $p(n)$ is the ordinary partition function and we agree that $p(x)=0$ when $x$ is not an integer.
Let $t_{k}$ be the reciprocal of 8 modulo $3^{k}$. Kolitsch \cite{Kolitsch-1} established the following infinite families of congruences: for $k \ge 1$ and $n\ge 0$,
\begin{align}\label{barcong}
\overline{c\phi}_{3}(3^{k}n+t_{k}) \equiv 0
\left\{\begin{array}{ll}
\pmod{3^{2k+2}} & \textrm{if $k$ is even},\\
\pmod{3^{2k+1}} & \textrm{if $k$ is odd}.
\end{array}\right.
\end{align}
From \eqref{bar-relate} we see that $c\phi_{3}(n)=\overline{c\phi}_{3}(n)$ if $n$ is not divisible by 3. Thus \eqref{barcong} implies that for $k\ge 1$ and $n\ge 0$,
\begin{align}
c\phi_{3}\Big(3^{2k-1}n+\frac{5\cdot 3^{2k-1}+1}{8}\Big) &\equiv 0 \pmod{3^{4k-1}}, \label{cphi3-cong1}\\
c\phi_{3}\Big(3^{2k}n+\frac{7\cdot 3^{2k}+1}{8} \Big) &\equiv 0 \pmod{3^{4k+2}}. \label{cphi3-cong2}
\end{align}
In 1996, by using some combinatorial arguments, Kolitsch \cite[Theorem 2]{Kolitsch-1996} proved that for any $n\ge 1$,
\begin{align}\label{cphi-9-bar-relation}
\overline{c\phi}_{9}(n)=3\overline{c\phi}_{3}(3n-1).
\end{align}
Using \eqref{bar-general}, this relation is equivalent to
\begin{align}\label{cphi-9-relate}
c\phi_{9}(n)=3c\phi_{3}(3n-1)+c\phi_{3}\left(\frac{n}{3} \right).
\end{align}
From \eqref{cphi3-cong1} and \eqref{cphi3-cong2}, Kolitsch \cite[Corollary 3]{Kolitsch-1996} used \eqref{cphi-9-bar-relation} to deduce the following congruences for $c\phi_{9}(n)$: for $k\ge 1$ and $n\ge 0$,
\begin{align}
c\phi_{9}\left(3^{2k}n+\frac{5\cdot 3^{2k}+3}{8} \right) &\equiv 0 \pmod{3^{4k-1}}, \label{cphi9-cong1}\\
c\phi_{9}\left(3^{2k+1}n+\frac{7\cdot 3^{2k+1}+3}{8} \right) &\equiv 0 \pmod{3^{4k+2}}.  \label{cphi9-cong2}
\end{align}

In 2011 and 2015, Baruah and Sarmah \cite{Baruah,Baruah-1}  found new representations of $\mathrm{C}\Phi_{k}(q)$ for $k\in \{4,5,6\}$. Moreover, they proved some congruences such as
\begin{align}
c\phi_{6}(3n+1) &\equiv 0 \pmod{9}, \label{cphi6-Baruah-1}\\
c\phi_{6}(3n+2) &\equiv 0 \pmod{9}. \label{cphi6-Baruah-2}
\end{align}
In 2016, Gu, Wang and Xia \cite{GuWangXia} found many congruences modulo powers of 3 for $c\phi_{6}(n)$. For example, for any integer $n\ge 0$, we have that
\begin{align}
c\phi_{6}(27n+16) &\equiv 0 \pmod{3^{5}}, \label{cphi6-1} \\
c\phi_{6}(243n+142)&\equiv 0 \pmod{3^6}. \label{cphi6-2}
\end{align}

Recently, Chan, Wang and Yang \cite{CWY-2} used the theory of modular forms to give many new representations for $\mathrm{C}\Phi_{k}(q)$  for $k\le 17$.  In particular, a $q$-product representation for $\mathrm{C}\Phi_{9}(q)$ was discovered for the first time:
\begin{align}\label{cphi9-gen}
\mathrm{C}\Phi_{9}(q)=\frac{E_{1}^3}{E_{3}^4}-240q\frac{E_{9}^3}{E_{3}^4}+324q\frac{E_{3}^8}{E_{1}^9}-1458q^2\frac{E_{9}^6}{E_{1}^3E_{3}^4}+19683q^4\frac{E_{9}^{12}}{E_{1}^9E_{3}^4},
\end{align}
where we denote
\[E_{k}=(q^k;q^k)_{\infty}, \quad k \in \mathbb{N}\]
for convenience. This representation leads to the congruences \cite[Theorem 5.2]{CWY-2}:
\begin{align}
c\phi_{9}(9n+3) &\equiv 0 \pmod{9}, \label{odd-cong1}\\
c\phi_{9}(9n+6) &\equiv 0 \pmod{9}, \label{odd-cong2}\\
c\phi_{9}(3n+1) &\equiv 0 \pmod{81}, \label{odd-cong3} \\
c\phi_{9}(3n+2) &\equiv 0 \pmod{729}. \label{odd-cong4}
\end{align}
For more results on $k$-colored generalized Frobenius partitions, see \cite{Baruah,Baruah-1}, \cite{CWY}--\cite{Sellers-2},
    \cite{Xia-0}--\cite{Zhang}.

Following their steps, we are going to present more congruences modulo powers of 3 for $c\phi_{3}(n)$ and $c\phi_{9}(n)$. Observe that the cases $k=1$ and 2 of \eqref{cphi3-cong1} and \eqref{cphi3-cong2} give the following congruences:
\begin{align}
c\phi_{3}(3n+2) &\equiv 0 \pmod{3^3}, \label{cong-1}\\
c\phi_{3}(9n+8) &\equiv 0 \pmod{3^6}, \label{cong-2}\\
c\phi_{3}(27n+17) &\equiv 0 \pmod{3^7}. \label{cong-3}
\end{align}
However, numerical evidences suggest that congruences \eqref{cong-2} and \eqref{cong-3} can be improved to
\begin{align}
c\phi_{3}(9n+8) &\equiv 0 \pmod{3^7}, \label{newcong-2}\\
c\phi_{3}(27n+17) &\equiv 0 \pmod{3^8}. \label{newcong-3}
\end{align}
Moreover, there are some congruences which are not included in \eqref{cphi3-cong1} and \eqref{cphi3-cong2}. For example, it appears that for any $n\ge 0$,
\begin{align}
c\phi_{3}(9n+5) &\equiv 0 \pmod{3^5}, \\
c\phi_{3}(27n+26) &\equiv 0 \pmod{3^9}.
\end{align}

The above observations motivate us to improve and extend Kolitsch's congruences \eqref{cphi3-cong1} and \eqref{cphi3-cong2}. First, we find some congruences modulo small powers of 3.
\begin{theorem}\label{thm-specific}
For any integer $n\ge 0$ we have
\begin{align}
c\phi_{3}(3n+1) &\equiv 0 \pmod{3^2}, \label{thmcong-1} \\
c\phi_{3}(3n+2) &\equiv 0 \pmod{3^3}, \label{thmcong-2} \\
c\phi_{3}(9n+5) &\equiv 0 \pmod{3^5}, \label{thmcong-3} \\
c\phi_{3}(9n+8) &\equiv 0 \pmod{3^7}, \label{thmcong-4} \\
c\phi_{3}(27n+17) &\equiv 0 \pmod{3^8}, \label{thmcong-5} \\
c\phi_{3}(27n+26) &\equiv 0 \pmod{3^9}, \label{thmcong-6} \\
c\phi_{3}(81n+44) &\equiv 0 \pmod{3^{10}}, \label{thmcong-7} \\
c\phi_{3}(81n+71) &\equiv 0 \pmod{3^{12}}, \label{thmcong-8} \\
c\phi_{3}(243n+152) &\equiv 0 \pmod{3^{13}}, \label{thmcong-9} \\
c\phi_{3}(243n+233) &\equiv 0 \pmod{3^{14}}. \label{thmcong-10}
\end{align}
\end{theorem}
All the moduli in congruences \eqref{thmcong-1}--\eqref{thmcong-10} cannot be replaced by higher powers of 3.

Then we give some general congruences beyond those in Theorem \ref{thm-specific}.
\begin{theorem}\label{thm-general}
For $k\ge 3$ and $n\ge 0$ we have
\begin{align}
c\phi_{3}\Big(3^{2k}n+\frac{7\cdot 3^{2k}+1}{8} \Big) &\equiv 0 \pmod{3^{4k+5}}, \label{general-1}\\
c\phi_{3}\Big(3^{2k+1}n+\frac{5\cdot 3^{2k+1}+1}{8}  \Big) &\equiv 0 \pmod{3^{4k+6}}, \label{general-2}\\
c\phi_{3}\Big(3^{2k+1}n+\frac{23\cdot 3^{2k}+1}{8}\Big) &\equiv 0 \pmod{3^{4k+7}}, \label{general-3}\\
c\phi_{3}\Big(3^{2k+2}n+\frac{13\cdot 3^{2k+1}+1}{8}\Big) &\equiv 0 \pmod{3^{4k+8}}. \label{general-4}
\end{align}
\end{theorem}

The second goal of this paper is to extend congruences \eqref{odd-cong1} and \eqref{odd-cong2} to larger families of congruences for $c\phi_{9}(n)$ and improve Kolitsch's congruences \eqref{cphi9-cong1} and \eqref{cphi9-cong2}. We obtain results for $c\phi_{9}(n)$ which are similar to Theorems \ref{thm-specific} and \ref{thm-general}.
\begin{theorem}\label{cphi9-thm-specific}
For any integer $n\ge 0$ we have
\begin{align}
c\phi_{9}(9n+3) &\equiv 0 \pmod{3^2}, \label{cphi9-specific-1}\\
c\phi_{9}(9n+6) &\equiv 0 \pmod{3^3}, \label{cphi9-specific-2}\\
c\phi_{9}(27n+15) &\equiv 0 \pmod{3^5}, \label{cphi9-specific-3}\\
c\phi_{9}(27n+24) &\equiv 0 \pmod{3^7}, \label{cphi9-specific-4}\\
c\phi_{9}(81n+51) &\equiv 0 \pmod{3^8}, \label{cphi9-specific-5}\\
c\phi_{9}(81n+78) &\equiv 0 \pmod{3^9}, \label{cphi9-specific-6}\\
c\phi_{9}(243n+132) &\equiv 0 \pmod{3^{10}}, \label{cphi9-specific-7}\\
c\phi_{9}(243n+213) &\equiv 0 \pmod{3^{12}}, \label{cphi9-specific-8}\\
c\phi_{9}(729n+456)&\equiv 0 \pmod{3^{13}}, \label{cphi9-specific-9}\\
c\phi_{9}(729n+699)&\equiv 0 \pmod{3^{14}}. \label{cphi9-specific-10}
\end{align}
\end{theorem}
All the moduli in congruences \eqref{cphi9-specific-1}--\eqref{cphi9-specific-10} cannot be replaced by higher powers of 3.

\begin{theorem}\label{cphi9-thm-general}
For $k\ge 3$ and $n\ge 0$ we have
\begin{align}
c\phi_{9}\Big(3^{2k+1}n+\frac{7\cdot 3^{2k+1}+3}{8} \Big) &\equiv 0 \pmod{3^{4k+5}}, \label{cphi9-general-1}\\
c\phi_{9}\Big(3^{2k+2}n+\frac{5\cdot 3^{2k+2}+3}{8}  \Big) &\equiv 0 \pmod{3^{4k+6}}, \label{cphi9-general-2}\\
c\phi_{9}\Big(3^{2k+2}n+\frac{23\cdot 3^{2k+1}+3}{8}\Big) &\equiv 0 \pmod{3^{4k+7}}, \label{cphi9-general-3}\\
c\phi_{9}\Big(3^{2k+3}n+\frac{13\cdot 3^{2k+2}+3}{8}\Big) &\equiv 0 \pmod{3^{4k+8}}. \label{cphi9-general-4}
\end{align}
\end{theorem}

The paper is organized as follows. In Section \ref{sec-cphi-3} we first collect some useful facts from the work of Kolitsch \cite{Kolitsch-1}. Then we give proofs to Theorems \ref{thm-specific}--\ref{thm-general}. In Section \ref{sec-cphi-9} we present two different proofs to the congruences satisfied by $c\phi_{9}(n)$. The first proof uses \eqref{cphi-9-relate} and Theorem \ref{thm-specific}--\ref{thm-general}. We will also give a new proof to \eqref{cphi-9-relate} by comparing the generating functions of $c\phi_{3}(n)$ and $c\phi_{9}(n)$. In the second proof of Theorems \ref{cphi9-thm-specific}--\ref{cphi9-thm-general},  we do not use \eqref{cphi-9-relate}. Instead, we will establish congruences modulo powers of 3 for the coefficients in the series expansion of each term in \eqref{cphi9-gen}.

\section{Congruences Modulo Powers of 3 for $c\phi_{3}(n)$}\label{sec-cphi-3}
Following the notation in \cite{Kolitsch-1}, we define
\begin{align*}
\xi=\frac{E_{1}^{3}}{qE_{9}^{3}}, \quad T=\frac{q^3E_{9}^{3}}{E_{3}^{3}}, \quad S=\frac{E_{3}^{4}}{E_{9}^{4}}+9\frac{q^3E_{3}E_{27}^3}{E_{9}^{4}}.
\end{align*}
The functions $\xi$ and $T$ satisfy an modular equation of order 3.
\begin{lemma}\label{modular}
(Cf.\ \cite[Lemma 3]{Kolitsch-1}) We have
\begin{align}
\xi^3+9\xi^2+27\xi=q^9T^{-4},
\end{align}
or equivalently,
\begin{align}
E_{1}^9E_{9}^3+9qE_{1}^6E_{9}^6+27q^2E_{1}^3E_{9}^9=E_{3}^{12}. \label{modeq}
\end{align}
\end{lemma}
Next we define an infinite matrix $\left(m_{i,j}\right)_{i,j \ge 1}$ by \\
(1) $m_{1,1}=3$, $m_{1,i}=0$ for $i\ge 2$; \\
(2) $m_{2,1}=1$, $m_{2,2}=3^4$, $m_{2,j}=0$ for $j\ge 3$;\\
(3) $m_{3,1}=0$, $m_{3,2}=2\cdot 3^3$, $m_{3,3}=3^7$, $m_{3,j}=0$ for $j\ge 4$; \\
(4) $m_{i,j}=m_{i-3,j-1}+9m_{i-2,j-1}+27m_{i-1,j-1}$ for $i\ge 4$, $j\ge 2$.

By induction on $i$, it is not difficult to show that
\begin{align}\label{m-vanish}
m_{i,j}=0,\quad \textrm{if} \quad i \ge 3j.
\end{align}
This was first observed by Kolitsch \cite[p.\ 346]{Kolitsch-1}.

We define operators $H_{m,r}$ ($0\le r \le m-1$) which act on infinite series as
\begin{align}
H_{m,r}\left(\sum_{n \in \mathbb{Z}}g(n)q^n \right):=\sum_{n \in \mathbb{Z}}g(mn+r)q^{mn+r}.
\end{align}
The importance of $\left(m_{i,j}\right)_{i,j \ge 1}$ can be seen from the following lemma.
\begin{lemma}\label{H2-xi-lem}
(Cf.\ \cite[Lemma 4]{Kolitsch-1}) We have
\begin{align}\label{H2-xi}
H_{3,2}\left(\xi^{-i} \right)=\frac{S}{q}\sum_{j=1}^{\infty}m_{i,j}T^{4j}q^{-9j}.
\end{align}
\end{lemma}
This lemma can be proved by using Lemma \ref{modular}. See \cite[p.\ 346]{Kolitsch-1} for discussions.

Next we define two matrices $\left(a_{i,j}\right)_{i,j\ge 1}$ and $\left(b_{i,j}\right)_{i,j\ge 1}$ by
\begin{align}
a_{i,j}=9m_{4i+1,i+j}+m_{4i,i+j}, \label{a-defn} \\
b_{i,j}=m_{4i-1,i+j}+9m_{4i,i+j}. \label{b-defn}
\end{align}
From direct computations, we find that
\begin{align*}
&a_{1,1}=21, \quad a_{1,2}=10206, \quad a_{1,3}=767637, \quad a_{1,4}=14348907, \quad a_{1,j}=0, \quad j \ge 5, \\
&b_{1,1}=162, \quad b_{1,2}=21870,  \quad b_{1,3}=531441,\quad b_{1,j}=0, \quad j \ge 4.
\end{align*}

\begin{lemma}\label{H2-lem}
(Cf. \cite[Lemma 5]{Kolitsch-1}.) We have
\begin{align*}
H_{3,2}\left(9\xi^{-4i-1}+\xi^{-4i}\right)&=\frac{S}{q}\sum_{j=1}^{\infty}a_{i,j}T^{4i+4j}q^{-9i-9j},\\
H_{3,2}\left(\xi^{-4i+1}+9\xi^{-4i}\right)&=\frac{S}{q}\sum_{j=1}^{\infty}b_{i,j}T^{4i+4j}q^{-9i-9j}.
\end{align*}
\end{lemma}

Now we define $\left(x_{i,j}\right)_{i,j\ge 1}$ by
\begin{align}\label{x-1st}
x_{1,1}=3, \quad x_{1,j}=0, \quad j\ge 2
\end{align}
and for $k\ge 1$,
\begin{align}
x_{2k,j}&=\sum_{i=1}^{\infty}x_{2k-1,i}b_{i,j}, \\
x_{2k+1,j}&=\sum_{i=1}^{\infty}x_{2k,i}a_{i,j}.
\end{align}

In order to prove \eqref{barcong}, Kolitsch \cite[Theorem 2]{Kolitsch-1} established formulas for $\sum_{n=0}^{\infty}\overline{c\phi}_{3}(3^{k}n+t_{k})q^n$. In view of \eqref{bar-relate}, his formulas can be restated as the following lemma.
\begin{lemma}\label{Kolitsch-gen}
For $k\ge 1$, we have
\begin{align}\label{odd}
\sum_{n=0}^{\infty}c\phi_{3}\Big(3^{2k-1}n+\frac{5\cdot 3^{2k-1}+1}{8}\Big)q^n =\frac{9T}{q^3E_{3}}\sum_{j=1}^{\infty}x_{2k-1,j}T^{-4j}q^{9j}\xi^{-4j}(\xi+9)
\end{align}
and
\begin{align}\label{even}
\sum_{n=0}^{\infty}{c\phi}_{3}\Big(3^{2k}n+\frac{7\cdot 3^{2k}+1}{8}\Big)q^n =\frac{9}{qE_{3}}\sum_{j=1}^{\infty}x_{2k,j}T^{-4j}q^{9j}\xi^{-4j-1}(\xi+9).
\end{align}
\end{lemma}

For any integer $n$, let $\pi(n)$ denote the 3-adic order of $n$ and we agree that $\pi(0)=\infty$. For any real number $x$, we denote by $[x]$ the integer part of $x$. In order to establish the congruences \eqref{cphi3-cong1} and \eqref{cphi3-cong2}, Kolitsch \cite{Kolitsch-1} examined the 3-adic orders of $x_{i,j}$ for $i,j\ge 1$.
\begin{lemma}\label{Kolitsch-adic}
(Cf. \cite[Lemma 9]{Kolitsch-1}.) We have
\begin{align}
\pi(x_{1,1})&=1, \\
\pi(x_{2k,j})&\ge 4k+\left[\frac{9j-9}{2}\right], \\
\pi(x_{2k+1,j})&\ge 4k+1+\left[\frac{9j-8}{2}\right].
\end{align}
\end{lemma}
It is then clear that when $j\ge 2$,
\begin{align}\label{xk2-order}
\pi(x_{2k,j})\ge 4k+4, \quad \pi(x_{2k+1,j})\ge 4k+6.
\end{align}
To prove Theorems \ref{thm-specific} and \ref{thm-general}, we need to improve the estimate for $\pi(x_{k,1})$.
\begin{lemma}\label{refine-lem}
For $k\ge 3$, we have
\begin{align}
\pi(x_{2k,1})&\ge 4k+3, \label{x-even}\\
\pi(x_{2k+1,1})&\ge 4k+4. \label{x-odd}
\end{align}
\end{lemma}
\begin{proof}
By computations we find that
\[a_{2,1}=1, \quad b_{2,1}=30, \quad b_{3,1}=1.\]
By \eqref{a-defn} we have
\begin{align}
a_{i,1}=9m_{4i+1,i+1}+m_{4i,i+1}.
\end{align}
By \eqref{m-vanish} we know that $a_{i,1}=0$ if $i\ge 3$.
Similarly by \eqref{b-defn} we have
\begin{align}
b_{i,1}=m_{4i-1,i+1}+9m_{4i,i+1}.
\end{align}
If $i\ge 4$, then by \eqref{m-vanish} we know $b_{i,1}=0$.

By definition we have
\begin{align}\label{x-rec1}
x_{2k,1}=162x_{2k-1,1}+30x_{2k-1,2}+x_{2k-1,3}
\end{align}
and
\begin{align}\label{x-rec2}
x_{2k+1,1}=21x_{2k,1}+x_{2k,2}.
\end{align}
Hence $x_{2,1}=2\cdot 3^5$ and
\begin{align}\label{x31}
\pi(x_{3,1})\ge \min\{1+\pi(x_{2,1}),\pi(x_{2,2})\}.
\end{align}
From \eqref{xk2-order} we see that $\pi(x_{2,2})\ge 8$. Hence \eqref{x31} implies $\pi(x_{3,1})\ge 6$. Next, using Lemma \ref{Kolitsch-adic} and \eqref{x-rec1} we see that
\begin{align}\label{x41}
\pi(x_{4,1})\ge \min\{4+\pi(x_{3,1}), 1+\pi(x_{3,2}), \pi(x_{3,3})\}\ge 10.
\end{align}
In the same way, we can prove that $\pi(x_{5,1})\ge 11$ and $\pi(x_{6,1})\ge 15$. Thus \eqref{x-even} is true for $k=3$.

Suppose \eqref{x-even} is true for some $k\ge 3$.
By \eqref{x-rec2} and \eqref{xk2-order} we deduce that
\begin{align*}
\pi(x_{2k+1,1})\ge \min\{1+\pi(x_{2k,1}),\pi(x_{2k,2})\} \ge 4k+4.
\end{align*}
Hence \eqref{x-odd} is true for $k$. By \eqref{x-rec1} and Lemma \ref{Kolitsch-adic} we deduce that
\begin{align*}
\pi(x_{2k+2,1})\ge \min\{4+\pi(x_{2k+1,1}), 1+\pi(x_{2k+1,2}), \pi(x_{2k+1,3})\} \ge 4k+7.
\end{align*}
This implies that \eqref{x-even} is true for $k+1$. By induction on $k$ we complete the proof of Lemma \ref{refine-lem}.
\end{proof}

Now we are able to prove Theorems \ref{thm-specific} and \ref{thm-general}.
\begin{proof}[Proof of Theorem \ref{thm-specific}]

From \cite[Theorem 1]{Kolitsch} we have
\[\overline{c\phi}_{3}(n)\equiv 0 \pmod{3^2}.\]
Since $c\phi_{3}(n)=c\phi_{3}(n)$ when $n$ is not divisible by 3, congruence \eqref{thmcong-1} follows.

Next,  letting $k=1$ in \eqref{odd}, from \eqref{x-1st} we deduce that
\[c\phi_{3}(3n+2) \equiv 0 \pmod{3^3}.\]
By \eqref{odd} and the binomial theorem, we have
\begin{align}
\sum_{n=0}^{\infty}c\phi_{3}(3n+2)q^n&=27\Big(\frac{E_{3}^8}{E_{1}^9}+9q\frac{E_{3}^8E_{9}^3}{E_{1}^{12}} \Big) \label{cphi-3-3n2} \\
&\equiv 27E_{3}^5 \pmod{3^5}.
\end{align}
Since the terms of the form $q^{3n+1}$ do not appear in the series expansion of $E_{3}^5$, we deduce that
\[c\phi_{3}(9n+5) \equiv 0 \pmod{3^5}.\]

In the proof of Lemma \ref{refine-lem} we have seen that $x_{2,1}=2\cdot 3^{5}$.  Moreover by \eqref{xk2-order} we know that $\pi(x_{2,j}) \ge 8$ for $j\ge 2$. Therefore by letting $k=1$ in \eqref{even} we deduce that
\[\sum_{n=0}^{\infty}c\phi_{3}(9n+8)q^n\equiv 2\cdot 3^7\Big( \frac{E_{3}^{11}}{E_{1}^{12}}+9q\frac{E_{3}^{11}E_{9}^{3}}{E_{1}^{15}}\Big) \pmod{3^{10}}.\]
It follows that
\[c\phi_{3}(9n+8) \equiv 0 \pmod{3^7}.\]
Moreover, by the binomial theorem, we deduce that
\begin{align}\label{cphi9n8}
\sum_{n=0}^{\infty}c\phi_{3}(9n+8)q^n\equiv 2\cdot 3^7\frac{E_{3}^{11}}{E_{1}^{12}}\equiv 2\cdot 3^7 \frac{E_{3}^8}{E_{1}^{3}} \pmod{3^{9}}.
\end{align}
From \cite[Lemma 2.6]{WangIJNT} we find
\begin{align}\label{3-dissection}
\frac{1}{E_{1}^3}=\frac{E_{9}^3}{E_{3}^{12}}\Big(E_{3}^2a^2(q^3)+3qE_{3}a(q^3)E_{9}^3+9q^2E_{9}^6\Big),
\end{align}
where
\begin{align}\label{aq-defn}
a(q)=\Big(1+6\sum_{n=0}^{\infty}\big(\frac{q^{3n+1}}{1-q^{3n+1}}-\frac{q^{3n+2}}{1-q^{3n+2}} \big)\Big).
\end{align}
Using \eqref{3-dissection}, extracting the terms on both sides of \eqref{cphi9n8} in which the exponent of $q$ is congruent to 2 modulo 3, we deduce that
\[c\phi_{3}(27n+26) \equiv 0 \pmod{3^9}.\]
Similarly, from \eqref{cphi9n8} and \eqref{3-dissection} we deduce that
\[c\phi_{3}(27n+17) \equiv 0 \pmod{3^{8}}.\]

In the proof of Lemma \ref{refine-lem} we have seen that $\pi(x_{3,1})\ge 6$. This together with \eqref{xk2-order} and \eqref{odd} imply that
\begin{align}\label{27n}
\sum_{n=0}^{\infty}c\phi_{3}(27n+17)q^n\equiv 9x_{3,1}\Big(\frac{E_{3}^8}{E_{1}^9}+9q\frac{E_{3}^8E_{9}^3}{E_{1}^{12}}  \Big) \pmod{3^{12}}.
\end{align}
By the binomial theorem, we deduce that
\begin{align}
\sum_{n=0}^{\infty}c\phi_{3}(27n+17)q^n\equiv \frac{E_{3}^8}{E_{1}^9} \equiv 9x_{3,1}E_{3}^5 \pmod{3^{10}}.
\end{align}
Hence by extracting the terms in which the exponent of $q$ is congruent to 1 modulo 3, we obtain
\[c\phi_{3}(81n+44) \equiv 0 \pmod{3^{10}}.\]

Let $k=2$ in \eqref{even}. By \eqref{xk2-order}, \eqref{x41} and \eqref{3-dissection} we deduce that
\begin{align}
\sum_{n=0}^{\infty}c\phi_{3}(81n+71)q^n &\equiv 9x_{4,1}\Big(\frac{E_{3}^{11}}{E_{1}^{12}}+9q\frac{E_{3}^{11}E_{9}^3}{E_{1}^{15}} \Big) \nonumber\\
&\equiv 9x_{4,1}\frac{E_{3}^{11}}{E_{1}^{9}}\cdot \frac{1}{E_{1}^{3}} \nonumber \\
&\equiv 9x_{4,1}\frac{E_{9}^3}{E_{3}^4}\Big(E_{3}^2a^2(q^3)+3qE_{3}a(q^3)E_{9}^3+9q^2E_{9}^6\Big) \pmod{3^{14}}.\label{last}
\end{align}
It follows that
\[c\phi_{3}(81n+71) \equiv 0 \pmod{3^{12}}.\]
Moreover, extracting the terms in which the exponent of $q$ is congruent to 1 and 2 on both sides of \eqref{last}, we obtain \eqref{thmcong-9} and \eqref{thmcong-10}, respectively.
\end{proof}

\begin{proof}[Proof of Theorem \ref{thm-general}]
By Lemma \ref{Kolitsch-adic} we have $\pi(x_{2k,2})\ge 4k+4$ and
\[\pi(x_{2k,j})\ge 4k+9, \quad j\ge 3.\]
By \eqref{even} and \eqref{xk2-order} we deduce that
\begin{align}\label{even-start}
&\sum_{n=0}^{\infty}c\phi_{3}\Big(3^{2k}n+\frac{7\cdot 3^{2k}+1}{8} \Big)q^n \nonumber\\
\equiv & 9 x_{2k,1}\Big(\frac{E_{3}^{11}}{E_{1}^{12}}+9q\frac{E_{3}^{11}E_{9}^{3}}{E_{1}^{15}}\Big)+9x_{2k,2}q\frac{E_{3}^{23}}{E_{1}^{24}} \pmod{3^{4k+8}}.
\end{align}
For $k\ge 3$, by Lemma \ref{refine-lem} we have $\pi(x_{2k,1})\ge 4k+3$ and hence
\[c\phi_{3}\Big(3^{2k}n+\frac{7\cdot 3^{2k}+1}{8} \Big) \equiv 0 \pmod{3^{4k+5}}.\]
Moreover, by \eqref{even-start} and the binomial theorem, we deduce that
\begin{align*}
\sum_{n=0}^{\infty}c\phi_{3}\Big(3^{2k}n+\frac{7\cdot 3^{2k}+1}{8} \Big)q^n\equiv 9x_{2k,1}E_{3}^{7} \pmod{3^{4k+6}}.
\end{align*}
It follows that
\[c\phi_{3}\Big(3^{2k}(3n+1)+\frac{7\cdot 3^{2k}+1}{8} \Big) \equiv 0 \pmod{3^{4k+6}},\]
which proves \eqref{general-2}.

By \eqref{even-start} and the binomial theorem again, we have
\begin{align}\label{even-last}
\sum_{n=0}^{\infty}c\phi_{3}\Big(3^{2k}n+\frac{7\cdot 3^{2k}+1}{8} \Big)q^n\equiv 9x_{2k,1}\frac{E_{3}^{8}}{E_{1}^3}+9x_{2k,2}qE_{3}^{15} \pmod{3^{4k+7}}.
\end{align}
Now recalling \eqref{3-dissection} and extracting the terms on both sides of \eqref{even-last} in which the exponent of $q$ is congruent to 2 modulo 3, we obtain \eqref{general-3}.

By \eqref{xk2-order}, \eqref{x-odd} and \eqref{odd} we deduce that
\begin{align}\label{odd-start}
\sum_{n=0}^{\infty}c\phi_{3}\Big(3^{2k+1}n+\frac{5\cdot 3^{2k+1}+1}{8} \Big)q^n \equiv 9 x_{2k+1,1}\frac{E_{3}^8}{E_{1}^9} \equiv 9 x_{2k+1,1}E_{3}^5 \pmod{3^{4k+8}}.
\end{align}
Extracting the terms on both sides of \eqref{odd-start} in which the exponent of $q$ is congruent to 1 modulo 3, we obtain \eqref{general-4}.
\end{proof}

\section{Congruences Modulo Powers of 3 for $c\phi_{9}(n)$}\label{sec-cphi-9}

In this section, we present two different proofs for Theorems \ref{cphi9-thm-specific} and \ref{cphi9-thm-general}.

Our first proof is based on Kolitsch's relation \eqref{cphi-9-relate}. Kolitsch's original proof of \eqref{cphi-9-relate} uses combinatorial arguments. Here we will give a new proof by series manipulations.
\begin{lemma}\label{a-exp-lem}
We have
\begin{align}
a(q)=\frac{E_{1}^3}{E_{3}}+9q\frac{E_{9}^3}{E_{3}} \label{a-new-exp}
\end{align}
and
\begin{align}
a^3(q)=\frac{E_{1}^9}{E_{3}^3}+27q\frac{E_{3}^9}{E_{1}^3}. \label{a-cubic}
\end{align}
\end{lemma}
\begin{proof}
It was proved in \cite[Eqs.\ (2.3), (2.4), (2.11)]{BBG} that
\begin{align}
a(q)=3a(q^3)-2\frac{E_{1}^3}{E_{3}}, \label{a-1} \\
a(q)=a(q^3)+6q\frac{E_{9}^3}{E_{3}}. \label{a-2}
\end{align}
By eliminating $a(q^3)$ from \eqref{a-1} and \eqref{a-2}, we obtain \eqref{a-new-exp}. 

The identity \eqref{a-cubic} follows from \cite[Proposition 2.2, Theorem 2.3]{BBG}. It also follows by taking the cubic power on both sides of \eqref{a-new-exp} and then using \eqref{modeq} to do simplifications.
\end{proof}
\begin{proof}[Proof of \eqref{cphi-9-relate}]
It suffices to show that
\begin{align}
\sum_{n=0}^{\infty}c\phi_{9}(n)q^n=\sum_{n=0}^{\infty}c\phi_{3}\left(\frac{n}{3} \right)q^n+3\sum_{n=0}^{\infty}c\phi_{3}(3n-1)q^n. \label{relate-start}
\end{align}
Using \eqref{a-new-exp} and \eqref{gen} we obtain
\begin{align}
\sum_{n=0}^{\infty}c\phi_{3}(n)q^n=\frac{1}{E_3}+9q\frac{E_{9}^3}{E_{1}^3E_3}
\end{align}
Combining this identity with \eqref{cphi-3-3n2} we deduce that
\begin{align}
\sum_{n=0}^{\infty}c\phi_{3}\left(\frac{n}{3} \right)q^n+3\sum_{n=0}^{\infty}c\phi_{3}(3n-1)q^n=\frac{1}{E_{9}}+9q^3\frac{E_{27}^3}{E_{3}^3E_{9}}+81q\frac{E_{3}^8}{E_{1}^9}+3^6q^2\frac{E_{3}^8E_{9}^3}{E_{1}^{12}}. \label{sum-gen}
\end{align}
Multiplying by $E_{1}^9$ on both sides of \eqref{cphi9-gen} and \eqref{sum-gen}, we know that \eqref{relate-start} is equivalent to
\begin{align}
F_1(q)=F_2(q), \label{equivalent}
\end{align}
where
\begin{align}
F_1(q)&=\frac{E_{1}^{12}}{E_{3}^4}-240q\frac{E_{1}^9E_{9}^3}{E_{3}^4}+243qE_{3}^8-1458q^2\frac{E_{1}^6E_{9}^6}{E_{3}^4}+19683q^4\frac{E_{9}^{12}}{E_{3}^4}, \label{F1-defn} \\
F_2(q)&=\frac{E_{1}^9}{E_{9}}+9q^3\frac{E_{1}^9E_{27}^3}{E_{3}^3E_{9}}+3^6q^2\frac{E_{3}^8E_{9}^3}{E_{1}^3}. \label{F2-defn}
\end{align}
By \cite[Lemma 2.5]{WangIJNT} we have
\begin{align}\label{Jacobi-dissection}
E_{1}^3=E_3a(q^3)-3qE_{9}^3.
\end{align}
Substituting \eqref{3-dissection} and \eqref{Jacobi-dissection} into \eqref{F1-defn}, after simplification, we obtain
\begin{align}
H_{3,0}\left(F_1(q) \right)&=a^4(q^3)+2160q^3\frac{E_{9}^9}{E_{3}^3}a(q^3), \label{H30-F1}\\
H_{3,1}\left(F_1(q) \right)&=243qE_{3}^8-252q\frac{E_{9}^3}{E_{3}}a^3(q^3)+13122q^4\frac{E_{9}^{12}}{E_{3}^4}, \label{H31-F1} \\
H_{3,2}\left(F_1(q) \right)&=756q^2\frac{E_{9}^6}{E_{3}^2}a^2(q^3). \label{H32-F1}
\end{align}
Similarly, substituting \eqref{3-dissection} and \eqref{Jacobi-dissection} into \eqref{F2-defn}, after simplification, we obtain
\begin{align}
H_{3,0}\left(F_2(q) \right)&=\frac{E_{3}^3}{E_{9}}a^3(q^3)+9q^3\frac{E_{27}^3}{E_{9}}a^3(q^3)-27q^3E_{9}^8+2187q^3\frac{E_{9}^9}{E_{3}^3}a(q^3)-243q^6\frac{E_{9}^8E_{27}^3}{E_{3}^3}, \label{H30-F2}\\
H_{3,1}\left(F_2(q) \right)&=-9qE_{3}^2E_{9}^2a^2(q^3)-81q^4\frac{E_{9}^2E_{27}^3}{E_{3}}a^2(q^3)+6561q^4\frac{E_{9}^{12}}{E_{3}^4}, \label{H31-F2} \\
H_{3,2}\left(F_2(q) \right)&=27q^2E_{3}E_{9}^5a(q^3)+729q^2\frac{E_{9}^6}{E_{3}^2}a^2(q^3)+243q^5\frac{E_{9}^5E_{27}^3}{E_{3}^2}a(q^3). \label{H32-F2}
\end{align}
Comparing \eqref{H30-F1} with \eqref{H30-F2}, we see that $H_{3,0}(F_1(q))=H_{3,0}(F_2(q))$ is equivalent to
\begin{align}
\left(a^3(q^3)-27q^3\frac{E_{9}^9}{E_{3}^3} \right)\left(a(q^3)-\frac{E_{3}^3}{E_{9}}-9q^3\frac{E_{27}^3}{E_{9}} \right)=0,
\end{align}
which follows from \eqref{a-new-exp}.

Comparing \eqref{H31-F1} with \eqref{H31-F2} and replacing $q^3$ by $q$, after simplification, we see that $H_{3,1}(F_1(q))=H_{3,1}(F_2(q))$ is equivalent to
\begin{align}
-28E_{3}^3a^3(q)+27E_{1}^{9}+729q\frac{E_{3}^{12}}{E_{1}^3}+9qE_{3}^2E_{9}^3a^2(q)+E_{1}^3E_{3}^2a^2(q)=0.  \label{H31-equiv}
\end{align}
Using \eqref{a-new-exp}, we observe that
\begin{align}
9qE_{3}^2E_{9}^3a^2(q)+E_{1}^3E_{3}^2a^2(q)=E_{3}^3a^2(q)\left(\frac{E_{1}^3}{E_3}+9q\frac{E_{9}^3}{E_{3}} \right)=E_{3}^3a^3(q).
\end{align}
Therefore, \eqref{H31-equiv} is the same as
\begin{align}
-E_{3}^3a^3(q)+E_{1}^9+27q\frac{E_{3}^{12}}{E_{1}^3}=0,
\end{align}
which follows from \eqref{a-cubic}. Hence $H_{3,1}(F_1(q))=H_{3,1}(F_2(q))$.

Next, comparing \eqref{H32-F1} with \eqref{H32-F2}, we know that $H_{3,2}(F_1(q))=H_{3,2}(F_2(q))$ is equivalent to
\begin{align}
E_{3}a(q)=E_{1}^3+9qE_{9}^3,
\end{align}
which follows from \eqref{a-new-exp}.

Thus we have proved that $H_{3,r}(F_1(q))=H_{3,r}(F_{2}(q))$ for $r\in \{0,1,2\}$. Hence \eqref{equivalent} holds and our proof is complete.
\end{proof}
\begin{rem}
The relation \eqref{cphi-9-relate} can also be proved using the theory of modular forms. Let $q=e^{2\pi i \tau}$ with $\mathrm{Im} \, \tau>0$. We denote by $M_{k}(\Gamma_{0}(N))$ the space of modular forms with weight $k$ on $\Gamma_{0}(N)$. It is not difficult to show that both $F_1(q)$ and $F_2(q)$ are in $M_{4}(\Gamma_{0}(27))$. Since $\dim M_{4}(\Gamma_{0}(27))=12$. By checking that the first 12 coefficients of $F_1(q)$ and $F_2(q)$ agree with each other, we immediately prove that $F_1(q)=F_2(q)$.
\end{rem}
Now we are able to prove Theorems \ref{cphi9-thm-specific} and \ref{cphi9-thm-general}.
\begin{proof}[First Proof of Theorems \ref{cphi9-thm-specific} and \ref{cphi9-thm-general}]
From \eqref{cphi-9-relate} we have
\begin{align}
c\phi_{9}(3n)=3c\phi_{3}(9n-1)+c\phi_{3}(n). \label{cphi9-cphi3}
\end{align}
It is then clear that all the congruences in Theorems \ref{cphi9-thm-specific} and \ref{cphi9-thm-general} follow from the congruences in Theorems \ref{thm-specific} and \ref{thm-general}. For example, \eqref{cphi9-cphi3} implies
\begin{align}
&c\phi_{9}\left(3^{2k+1}n+\frac{7\cdot 3^{2k+1}+3}{8} \right)\nonumber \\
=&~3c\phi_{3}\left(3^{2k+2}n +\frac{7\cdot 3^{2k+2}+1}{8} \right)+c\phi_{3}\left(3^{2k}n+\frac{7\cdot 3^{2k}+1}{8} \right).
\end{align}
This together with \eqref{general-1} implies \eqref{cphi9-general-1}. Other congruences can be proved in a similar fashion.
\end{proof}

We can also give another proof without using Kolitsch's relation \eqref{cphi-9-relate}. For this we examine the terms in \eqref{cphi9-gen} one by one. We define five sequences $a_{i}(n)$ ($1\le i \le 5$) by
\begin{align}
\sum_{n=0}^{\infty}a_{1}(n)q^n&=\frac{E_{1}^3}{E_{3}^4}, \\
\sum_{n=0}^{\infty}a_{2}(n)q^n&=q\frac{E_{9}^3}{E_{3}^4}, \\
\sum_{n=0}^{\infty}a_{3}(n)q^n&=q\frac{E_{3}^8}{E_{1}^9},\\
\sum_{n=0}^{\infty}a_{4}(n)q^n&=q^2\frac{E_{9}^6}{E_{1}^3E_{3}^4},\\
\sum_{n=0}^{\infty}a_{5}(n)q^n&=q^4\frac{E_{9}^{12}}{E_{1}^9E_{3}^4}.
\end{align}
Then \eqref{cphi9-gen} implies
\begin{align}\label{cphi9-relation}
c\phi_{9}(n)=a_{1}(n)-240a_{2}(n)+324a_{3}(n)-1458a_{4}(n)+19683a_{5}(n).
\end{align}

By \eqref{Jacobi-dissection} we have
\begin{align}
\sum_{n=0}^{\infty}a_{1}(3n)q^n=\frac{a(q)}{E_{1}^3},
\end{align}
Comparing this with \eqref{gen}, we find that
\begin{align}\label{a-c-relation}
a_{1}(3n)=c\phi_{3}(n).
\end{align}
According to the definition, we have $a_{2}(3n)=0$. Hence
\begin{align}\label{cphi9-new}
c\phi_{9}(3n)=c\phi_{3}(n)+324a_{3}(3n)-1458a_{4}(3n)+19683a_{5}(3n).
\end{align}

Now we establish some results which are analogous to Lemma \ref{Kolitsch-gen}.
\begin{lemma}\label{a-gen}
For $k\ge 1$ we have\\
$(1)$
\begin{align}\label{a3-even}
\sum_{n=0}^{\infty}a_{3}\Big(3^{2k}n+\frac{5\cdot 3^{2k}+3}{8}\Big)q^n=9\frac{T}{q^3E_{3}}\sum_{j=1}^{\infty}y_{2k-1,j}T^{-4j}q^{9j}\xi^{-4j}(\xi+9)
\end{align}
and
\begin{align}\label{a3-odd}
\sum_{n=0}^{\infty}a_{3}\Big(3^{2k+1}n+\frac{7\cdot 3^{2k+1}+3}{8}\Big)q^n =9\frac{1}{qE_{3}}\sum_{j=1}^{\infty}y_{2k,j}T^{-4j}q^{9j}\xi^{-4j-1}(\xi+9),
\end{align}
where
\begin{align}\label{y-1st-defn}
y_{1,j}=6a_{1,j}+243a_{2,j}, \quad j \ge 1
\end{align}
and for $k \ge 1$,
\begin{align}\label{y-defn}
y_{2k,j}=\sum_{i=1}^{\infty}y_{2k-1,i}b_{i,j}, \quad  y_{2k+1,j}=\sum_{i=1}^{\infty}y_{2k,i}a_{i,j}.
\end{align}
$(2)$ \begin{align}\label{a4-even}
\sum_{n=0}^{\infty}a_{4}\Big(3^{2k}n+\frac{5\cdot 3^{2k}+3}{8}\Big)q^n=3q^3\frac{E_{9}^{15}}{E_{3}^{16}}\sum_{j=1}^{\infty}z_{2k-1,j}T^{-4j-4}q^{9j+9}\xi^{-4j}(\xi+9)
\end{align}
and for $k\ge 1$,
\begin{align}\label{a4-odd}
\sum_{n=0}^{\infty}a_{4}\Big(3^{2k+1}n+\frac{7\cdot 3^{2k+1}+3}{8}\Big)q^n =3\frac{E_{3}^{11}}{q^4E_{9}^{12}}\sum_{j=1}^{\infty}z_{2k,j}T^{4-4j}q^{9j-9}\xi^{-4j-1}(\xi+9),
\end{align}
 where
 \[z_{1,1}=21,\quad z_{1,2}=2\cdot 3^6 \cdot 7, \quad z_{1,3}=3^{10}\cdot 13, \quad z_{1,4}=3^{15}, \quad z_{1,j}=0, \quad j \ge 5\]
 and for $k\ge 1$,
 \begin{align}
 z_{2k,j}=\sum_{i=1}^{\infty}z_{2k-1,i}a_{i,j}, \quad  z_{2k+1,j}=\sum_{i=1}^{\infty}z_{2k,i}b_{i,j}.
 \end{align}
$(3)$
 \begin{align}\label{a5-even}
\sum_{n=0}^{\infty}a_{5}\Big(3^{2k}n+\frac{5\cdot 3^{2k}+3}{8}\Big)q^n=9q^3\frac{E_{9}^{15}}{E_{3}^{16}}\sum_{j=1}^{\infty}w_{2k-1,j}T^{-4j-4}q^{9j+9}\xi^{-4j}(\xi+9)
\end{align}
and
\begin{align}\label{a5-odd}
\sum_{n=0}^{\infty}a_{5}\Big(3^{2k+1}n+\frac{7\cdot 3^{2k+1}+3}{8}\Big)q^n =9\frac{E_{3}^{11}}{q^4E_{9}^{12}}\sum_{j=1}^{\infty}w_{2k,j}T^{4-4j}q^{9j-9}\xi^{-4j-1}(\xi+9),
\end{align}
 where
 \[w_{1,j}=6a_{2,j}+243a_{3,j}, \quad j \ge 1\]
 and
 \begin{align}
 w_{2k,j}=\sum_{i=1}^{\infty}w_{2k-1,i}a_{i,j}, \quad  w_{2k+1,j}=\sum_{i=1}^{\infty}w_{2k,i}b_{i,j}.
 \end{align}
\end{lemma}
\begin{proof}
The proofs of (1)-(3) are similar to the proof of Lemma \ref{Kolitsch-gen} given in \cite{Kolitsch-1}. Here we only give the details for (1).

We proceed by induction on $k$. We write
\[\sum_{n=0}^{\infty}a_{3}(n)q^{n+2}=\frac{E_{3}^8}{E_{9}^9}\xi^{-3}.\]
Applying the operator $H_{3,2}$ to both sides, we obtain
\[\sum_{n=0}^{\infty}a_{3}(3n)q^{3n}=\frac{E_{3}^8S}{q^3E_{9}^9}(54T^8q^{-18}+2187T^{12}q^{-27}).  \]
Replacing $q^3$ by $q$, using the fact that
\begin{align}\label{ST}
T(q^{1/3})=\frac{qE_{3}^3}{E_{1}^3}=T^{-1}\xi^{-1}q^3, \quad S(q^{1/3})=\frac{E_{1}^4}{E_{3}^4}(1+9\xi^{-1}),
\end{align}
we obtain
\begin{align}\label{gen-start}
\sum_{n=0}^{\infty}a_{3}(3n)q^n=\frac{9}{E_{3}}\left(6T^{-4}q^9\xi^{-5}+243T^{-8}q^{18}\xi^{-9}\right)(\xi+9).
\end{align}
Applying $H_{3,2}$ to both sides of \eqref{gen-start}, we get
\begin{align}
\sum_{n=0}^{\infty}a_{3}(9n+6))q^{3n+2}&=\frac{9}{E_{3}}\Big(6T^{-4}q^9H_{3,2}(9\xi^{-5}+\xi^{-4})+243T^{-8}q^{18}H_{3,2}(9\xi^{-9}+\xi^{-8})  \Big) \nonumber \\
&=\frac{9S}{qE_{3}}\Big(6T^{-4}q^9 \sum_{j=1}^{\infty}a_{1,j}T^{4+4j}q^{-9-9j}+243T^{-8}q^{18}\sum_{j=1}^{\infty}a_{2,j}T^{8+4j}q^{-18-9j}\Big) \nonumber \\
&=\frac{9S}{qE_{3}}\sum_{j=1}^{\infty}y_{1,j}T^{4j}q^{-9j}. \label{gen-mid}
\end{align}
Now dividing both sides by $q^2$ and replacing $q^3$ by $q$, using \eqref{ST} we obtain
\begin{align}
\sum_{n=0}^{\infty}a_{3}(9n+6)q^n=\frac{9T}{q^3E_{3}}\sum_{j=1}^{\infty}y_{1,j}T^{-4j}q^{9j}\xi^{-4j}(\xi+9).
\end{align}
This proves \eqref{a3-even} for $k=1$.

Suppose \eqref{a3-even} is true for some integer $k\ge 1$. Extracting the terms in which the exponent of $q$ is congruent to 2 modulo 3, we obtain
\begin{align}
\sum_{n=0}^{\infty}a_{3}\Big(3^{2k}(3n+2)+\frac{5\cdot 3^{2k}+3}{8}  \Big)q^{3n+2}=\frac{9T}{q^3E_{3}}\sum_{i=1}^{\infty}y_{2k-1,i}T^{-4i}q^{9i}H_{3,2}(\xi^{-4i+1}+9\xi^{-4i}).
\end{align}
Using Lemma \ref{H2-lem} it follows that
\begin{align}
\sum_{n=0}^{\infty}a_{3}\Big(3^{2k+1}n+\frac{7\cdot 3^{2k+1}+3}{8}\Big)q^{3n+2}=\frac{9TS}{q^4E_{3}}\sum_{j=1}^{\infty}y_{2k,j}T^{4j}q^{-9j}.
\end{align}
Dividing both sides by $q^2$ and then replacing $q^3$ by $q$, using \eqref{ST} we obtain \eqref{a3-odd}.

Next, extracting the terms in which the exponent of $q$ is congruent to 1 modulo 3 on both sides of \eqref{a3-odd}, we have
\begin{align}
&\sum_{n=0}^{\infty}c\phi_{3}\Big(3^{2k+1}(3n+1)+\frac{7\cdot 3^{2k+1}+3}{8} \Big)q^{3n+1} \nonumber \\ &=\frac{9}{qE_{3}}\sum_{i=1}^{\infty}y_{2k,i}T^{-4i}q^{9i}H_{3,2}(9\xi^{-4i-1}+\xi^{-4i}).
\end{align}
Using Lemma \ref{H2-lem} it follows that
\begin{align}
\sum_{n=0}^{\infty}c\phi_{3}\Big(3^{2k+2}n+\frac{5\cdot 3^{2k+2}+3}{8}\Big)q^{3n+1}=\frac{9S}{q^2E_{3}}\sum_{j=1}^{\infty}y_{2k+1,j}T^{4j}q^{-9j}.
\end{align}
Dividing both sides by $q$, then replacing $q^3$ by $q$ and using \eqref{ST}, we obtain
\begin{align}
\sum_{n=0}^{\infty}c\phi_{3}\Big(3^{2k+2}n+\frac{5\cdot 3^{2k+2}+3}{8}\Big)q^{n}=\frac{9T}{q^3E_{3}}\sum_{j=1}^{\infty}y_{2k+1,j}T^{-4j}q^{9j}\xi^{-4j}(\xi+9).
\end{align}
This gives \eqref{a3-even} with $k$ replaced by $k+1$. By induction on $k$ we complete the proof of (1).
\end{proof}

In order to establish congruences modulo arbitrary powers of 3 for $c\phi_{9}(n)$, we need to examine the 3-adic orders of $y_{i,j}$, $z_{i,j}$ and $w_{i,j}$. The following information about the 3-adics orders of $a_{i,j}$ and $b_{i,j}$ will be helpful.
\begin{lemma}\label{ab-adic}
(Cf. \cite[Lemma 8]{Kolitsch-1}.) We have
\[\pi(a_{i,j})\ge \left[\frac{9j-3i-3}{2}\right], \quad \pi(b_{i,j})\ge \left[\frac{9j-3i}{2}\right].\]
\end{lemma}
\begin{lemma}\label{yzw-adic}
For any integer $k\ge 1$ we have \\
$(1)$ \begin{align*}
\pi(y_{2k-1,j})\ge 4k-2+\left[\frac{9j-8}{2} \right], \quad \pi(y_{2k,j})\ge 4k+1+\left[\frac{9j-9}{2}\right].
\end{align*}
$(2)$ \begin{align*}
\pi(z_{2k-1,j})\ge 4k-3+\left[\frac{9j-8}{2}\right], \quad \pi(z_{2k,j})\ge 4k+\left[\frac{9j-9}{2}\right].
\end{align*}
$(3)$ \begin{align*}
\pi(w_{2k-1,j})\ge 4k-3+\left[\frac{9j-8}{2}\right], \quad \pi(w_{2k,j})\ge 4k+\left[\frac{9j-9}{2}\right].
\end{align*}
\end{lemma}
\begin{proof}
We only prove (1) in details. The proofs of (2) and (3) are similar.

To prove (1) we use induction on $k$. For $k=1$, by \eqref{y-1st-defn} and Lemma \ref{ab-adic} we find that
\begin{align}
\pi(y_{1,j})\ge \min \left\{1+\left[\frac{9j-6}{2} \right], 5+\left[\frac{9j-9}{2} \right]\right\}= 2+\left[\frac{9j-8}{2} \right].
\end{align}
Hence the first inequality in (1) is true for $k=1$.

Suppose the first inequality in (1) is true for some $k \ge 1$. From \eqref{y-defn} we have
\begin{align}\label{ineq-start}
\pi(y_{2k,j})\ge \min\limits_{i\ge 1} \pi(y_{2k-1,i}b_{i,j}) \ge \min\limits_{i\ge 1} \Big(4k-2+\left[\frac{9i-8}{2}\right]+\left[\frac{9j-3i}{2}\right]\Big).
\end{align}
Let
\[f(i,j)=\left[\frac{9i-8}{2}\right]+\left[\frac{9j-3i}{2}\right].\]
If we increase $i$ by 1, the value of $\left[\frac{9i-8}{2}\right]$ increases by at least 4 and the value of $\left[\frac{9j-3i}{2}\right]$ decreases by at most 2. Thus $f(i,j)$ is increasing with respect to $i$. So the minimal value of $f(i,j)$ occurs when $i=1$. Therefore, from \eqref{ineq-start} we obtain
\begin{align*}
\pi(y_{2k,j})\ge 4k-2+\left[\frac{9j-3}{2}\right]=4k+1+\left[\frac{9j-9}{2}\right].
\end{align*}
Hence the second inequality in (1) is true for $k$.

Similarly, we have
\begin{align*}
\pi(y_{2k+1,j})&\ge \min\limits_{i\ge 1} \Big(\pi(y_{2k,i})+\pi(a_{i,j})  \Big)\\
&\ge \min_{i\ge 1} 4k+1+\left[\frac{9i-9}{2}\right]+\left[\frac{9j-3i-3}{2}\right] \\
&\ge 4k+2+\left[\frac{9j-8}{2}\right].
\end{align*}
This proves the first inequality of (1) for $k+1$. By induction we know (1) is true for any $k\ge 1$.
\end{proof}

\begin{lemma}\label{ak-thm}
For any integers $k\ge 1$ and $n\ge 0$ we have
\begin{align}
a_{3}\Big(3^{2k}n+\frac{5\cdot 3^{2k}+3}{8}\Big) &\equiv 0 \pmod{3^{4k}}, \label{a3-even-cong} \\
a_{3}\Big(3^{2k+1}n+\frac{7\cdot 3^{2k+1}+3}{8}\Big) &\equiv 0 \pmod{3^{4k+3}}. \label{a3-odd-cong}\\
a_{4}\Big(3^{2k}n+\frac{5\cdot 3^{2k}+3}{8}\Big) &\equiv 0 \pmod{3^{4k-2}}, \label{a4-even-cong}\\
a_{4}\Big(3^{2k+1}n+\frac{7\cdot 3^{2k+1}+3}{8}\Big) &\equiv 0 \pmod{3^{4k+1}}. \label{a4-odd-cong}\\
a_{5}\Big(3^{2k}n+\frac{5\cdot 3^{2k}+3}{8}\Big) &\equiv 0 \pmod{3^{4k-1}}, \label{a5-even-cong}\\
a_{5}\Big(3^{2k+1}n+\frac{7\cdot 3^{2k+1}+3}{8}\Big) &\equiv 0 \pmod{3^{4k+2}}. \label{a5-odd-cong}
\end{align}
\end{lemma}
\begin{proof}
The assertions follow directly from Lemmas \ref{a-gen} and \ref{yzw-adic}.
\end{proof}

\begin{proof}[Proofs of Theorems \ref{cphi9-thm-specific} and \ref{cphi9-thm-general}]
Recall from \eqref{cphi9-new} that
\begin{align}\label{cphi9-repn}
c\phi_{9}(3n)=c\phi_{3}(n)+4\cdot 3^4 a_{3}(3n)-2\cdot 3^6a_{4}(3n)+3^9a_{5}(3n).
\end{align}
Congruences \eqref{cphi9-specific-1} and \eqref{cphi9-specific-2}  follow from \eqref{thmcong-1} and \eqref{thmcong-2}, respectively.

From \eqref{a3-odd-cong}, \eqref{a4-odd-cong}, \eqref{a5-odd-cong} and \eqref{cphi9-repn} we deduce that for any $k\ge 1$,
\begin{align}\label{final-relation1}
c\phi_{9}\Big(3\big(3^{2k}n+\frac{7\cdot 3^{2k}+1}{8} \big)\Big)\equiv c\phi_{3}\Big(3^{2k}n+\frac{7\cdot 3^{2k}+1}{8} \Big) \pmod{3^{4k+7}}.
\end{align}
Similarly, from \eqref{a3-even-cong}, \eqref{a4-even-cong}, \eqref{a5-even-cong} and \eqref{cphi9-repn} we deduce that for any $k\ge 0$,
\begin{align}\label{final-relation2}
c\phi_{9}\Big(3\big(3^{2k+1}n+\frac{5\cdot 3^{2k+1}+1}{8}\big)\Big)\equiv c\phi_{3}\Big(3^{2k+1}n+\frac{5\cdot 3^{2k+1}+1}{8}\Big) \pmod{3^{4k+8}}.
\end{align}
Theorem \ref{cphi9-thm-specific} then follows from \eqref{final-relation1}, \eqref{final-relation2} and Theorem \ref{thm-specific}. For example, let $k=0$ in \eqref{final-relation2}. We have
\begin{align}\label{mid-relation}
c\phi_{9}(9n+6) \equiv c\phi_{3}(3n+2) \pmod{3^8}.
\end{align}
Replacing $n$ by $3n+1$ and $3n+2$ in \eqref{mid-relation}, we obtain \eqref{cphi9-specific-3} and \eqref{cphi9-specific-4}, respectively. Congruences \eqref{cphi9-specific-5}--\eqref{cphi9-specific-10} can be proved similarly.

Congruence \eqref{cphi9-general-1} follows from \eqref{final-relation1} and \eqref{general-1}. Congruence \eqref{cphi9-general-2} follows from \eqref{final-relation2} and  \eqref{general-2}. Congruence \eqref{cphi9-general-3} follows after replacing $n$ by $3n+2$ in \eqref{final-relation1} and using \eqref{general-3}. Congruence \eqref{cphi9-general-4} follows after replacing $n$ by $3n+1$ in \eqref{final-relation2} and using  \eqref{general-4}.
\end{proof}
\begin{rem}
Although the first proof is much shorter, the second proof reveals more information. It establishes congruences modulo powers of 3 for the sequences $a_{i}(n)$ ($i\in \{3,4,5\}$), which may be useful for people who are interested in these sequences.
\end{rem}

\end{document}